\documentclass[11]{amsart}

\usepackage{amsfonts}
\usepackage{amscd}
\usepackage{amsmath, mathrsfs, amssymb}
\usepackage{amsthm}
\usepackage{setspace}
\usepackage{hyperref}
\usepackage{color}
\usepackage{epsfig}
\usepackage{here}
\usepackage{graphicx}
\usepackage[all]{xy}
\usepackage{psfrag}
\usepackage{graphicx,transparent}
\usepackage{enumerate}
\usepackage{caption}

\theoremstyle{plain}
\newtheorem{theorem}{Theorem}[section]
\newtheorem{lemma}[theorem]{Lemma}

\theoremstyle{definition}

\newcommand{\MM}{\mathcal M}

\newcommand{\BM}{\overline{\mathcal M}}

\newcommand{\QQ}{\mathcal Q}

\newcommand{\calF}{\mathcal F}
\newcommand{\calE}{\mathcal E}

\newcommand{\OO}{\mathcal O}

\newcommand{\BQQ}{\overline{\mathcal Q}}
\newcommand{\BFF}{\overline{\mathcal F}}

\newcommand{\irr}{\operatorname{irr}}

\newcommand{\Pic}{\operatorname{Pic}}

\newcommand{\bbP}{\mathbb P}

\newcommand{\bbQ}{\mathbb Q}
\newcommand{\bbR}{\mathbb R}

\newcommand{\SL}{\operatorname{SL}}

\newcommand{\calG}{\mathcal G}

\begin{document}

\makeatletter
	\@namedef{subjclassname@2020}{%
	\textup{2020} Mathematics Subject Classification}
	\makeatother

\title[Invariants on the flex and gothic loci]{Dynamical invariants and intersection theory on the flex and gothic loci}

\author{Dawei Chen}
\address{Department of Mathematics, Boston College, Chestnut Hill, MA 02467, USA}
\email{dawei.chen@bc.edu}

\subjclass[2020]{14H10, 14H15, 32G15}
\keywords{plane cubics, quadratic differentials, $\SL_2(\bbR)$-invariant varieties, Lyapunov exponents}

\date{\today}

\thanks{Research partially supported by the National Science Foundation Grant DMS-2001040 and the Simons Collaboration Grant 635235}

\begin{abstract}
The flex locus parameterizes plane cubics with three collinear cocritical points under a projection, and the gothic locus arises from quadratic differentials with zeros at a fiber of the projection and with poles at the cocritical points.  The flex and gothic loci provide the first example of a primitive, totally geodesic subvariety of  moduli space and new $\SL_2(\bbR)$-invariant varieties in Teichm\"uller dynamics, as discovered by McMullen--Mukamel--Wright. In this paper we determine the divisor class of the flex locus as well as various tautological intersection numbers on the gothic locus. For the case of the gothic locus our result confirms numerically a conjecture of Chen--M\"oller--Sauvaget about computing sums of Lyapunov exponents for $\SL_2(\bbR)$-invariant varieties via intersection theory.  
\end{abstract}

\maketitle


\section{introduction}
\label{sec:intro}

Let $\calF\subset \MM_{1,3}$ be the flex locus parameterizing elliptic curves with three marked points $(E, p_1, p_2, p_3)$ where $E$ admits a plane cubic model with a projection $\pi$ from a point $s\in \bbP^2$ such that $p_1, p_2, p_3$ are three collinear cocritical points of $\pi$. Here $p_i$ being a cocritical point means the corresponding fiber of $\pi$ is of type $p_i + 2q_i$ where $q_i$ is a critical point, i.e. $\pi$ is ramified at $q_i$.  Let $\QQ \calF \subset \QQ_{1,3}$ be the locus of quadratic differentials $(E, q)$ over $\calF$ such that ${\rm div}(q) = z_1 + z_2 + z_3 - p_1 - p_2 - p_3$ where $z_1 + z_2 + z_3$ is a fiber of $\pi$. Lifting via the canonical double cover identifies $\QQ \calF \subset \QQ_{1,3}$ with the gothic locus $\Omega \calG \subset \Omega\MM_{4}(2^3, 0^3)$ of abelian differentials on curves of genus four where the preimage of each $z_i$ is a double zero and the preimage of each $p_i$ is an ordinary point (i.e. a zero of order zero). Since $\QQ \calF \cong \Omega \calG$, we also refer to $\QQ \calF$ as the gothic locus.  

The flex and gothic loci were discovered by McMullen--Mukamel--Wright \cite{MMW}. They are not only interesting objects in classical algebraic geometry, but they also possess remarkable  
properties from the viewpoint of Teichm\"uller dynamics. The flex locus gives the first example of a primitive, totally geodesic subvariety of moduli space, and the gothic locus provides new $\SL_2(\bbR)$-invariant varieties. In this paper we focus on another fascinating interplay by using intersection theory to study dynamical invariants on these loci.  

Let $\BFF\subset \BM_{1,3}$ be the closure of $\calF$ in the Deligne-Mumford compactification. Our first result determines the divisor class of $\BFF$ in terms of the standard generators of 
$\Pic_{\bbQ}(\BM_{1,3})$. 

\begin{theorem}
\label{thm:divisor}
The divisor class of $\BFF$ in $\BM_{1,3}$ is 
\begin{eqnarray}
\label{eq:F}
 \BFF = \frac{4}{3} \delta_{\irr} + 4 \big(\delta_{0; \{1,2\}} + \delta_{0; \{1,3\}} + \delta_{0; \{2,3\}}\big) + 4 \delta_{0; \{1,2, 3\}}. 
 \end{eqnarray}
\end{theorem}

Let $\BQQ_{1,3}$ be the (twisted) quadratic Hodge bundle on $\BM_{1,3}$ whose fiber over $(E, p_1, p_2, p_3)$ is $H^0(E, K^{\otimes 2}(p_1+p_2+p_3))$, i.e. the space of quadratic differentials $q$ with at worst simple poles at the $p_i$, where $K$ is the dualizing line bundle of the stable curve $E$. Denote by $\QQ\BFF$ the closure of the gothic locus $\QQ \calF$ in $\BQQ_{1,3}$. We use $\bbP \QQ\BFF\subset \bbP\BQQ_{1,3}$ to denote their projectivizations. 

Let $\OO(-1)$ be the tautological bundle on $\bbP\BQQ_{1,3}$ whose fiber over $(E, [q], p_1, p_2, p_3)$ is spanned by $q$, and let $\eta$ be the first Chern class of $\OO(-1)$. In \cite[Conjecture 4.3]{CMS} there was a conjectural formula about the sum of Lyapunov exponents for an arbitrary $\SL_2(\bbR)$-invariant variety in terms of intersection numbers of divisor classes $\eta$, $\lambda_1$, and $\psi_i$, where $\lambda_1$ is the first Chern class of the (ordinary) Hodge bundle and $\psi_i$ is the cotangent line bundle class associated to the $i$-th marked point. For the gothic locus as an 
$\SL_2(\bbR)$-invariant variety, all of its periods are absolute, hence no $\psi$-classes are involved.  In this case the conjecture reduces to 
$$L^{+}(\QQ\calF) = 2 \frac{\int_{\bbP\QQ\BFF} \eta^2 \lambda_1}{\int_{\bbP\QQ\BFF} \eta^3} $$
where the factor $2$ is due to the canonical double cover when lifting to the actual gothic locus $\Omega \calG$ as explained in \cite{CMS}, and the positive sign here is because we work with $\OO(-1)$ instead of its dual bundle in~\cite{CMS}.  

Our next result determines the above intersection numbers.     

\begin{theorem}
\label{thm:L} 
We have $\int_{\bbP\QQ\BFF} \eta^2 \lambda_1 = -1/2$, $\int_{\bbP\QQ\BFF} \eta^3 = -13/6$, and 
\begin{eqnarray}
\label{eq:L} 
 2 \frac{\int_{\bbP\QQ\BFF} \eta^2 \lambda_1}{\int_{\bbP\QQ\BFF} \eta^3} =  \frac{6}{13}. 
 \end{eqnarray}
\end{theorem}

We remark that the value $6/13$ matches with computer experiments by using square-tiled surfaces of large degree $d$ in the gothic locus (shared with the author by M\"oller). For instance for $d = 188$, one approximation gives $0.46155\ldots$ while $6/13 \approx 0.46153\ldots$ Therefore, our result provides a numerical confirmation of \cite[Conjecture 4.3]{CMS} for the gothic locus. Moreover, the Masur--Veech volume of the gothic locus was computed by Torres-Teigell~\cite{T} via counting square-tiled surfaces. Note that the denominator $13$ in~\eqref{eq:L} also appears in~\cite[Theorem 1.1]{T}. This can be explained by the fact that the top intersection number of $\eta$ on an $\SL_2(\bbR)$-invariant variety with no REL deformation corresponds to the Masur--Veech volume (up to a volume normalization factor, see~\cite[Section 2]{CMS} and \cite[Section 3]{KN}). 

For recent developments about computing invariants in Teichm\"uller dynamics via intersection theory, such as Masur--Veech volumes and (area) Siegel--Veech constants, see~\cite{CMSZ, CMS} as well as the references therein. Several other special $\SL_2(\bbR)$-invariant varieties analogous to the gothic locus were discovered by Eskin--McMullen--Mukamel--Wright \cite{EMMW} which we plan to study in future work.  

\subsection*{Acknowledgements} The author is grateful to Martin M\"oller for sharing the data of Lyapunov exponents of square-tiled surfaces as well as related discussions.  

\section{Divisor class of the flex locus}
\label{sec:divisor}

We first study the intersection of the flex locus $\BFF$ with the boundary divisors of $\BM_{1,3}$. In order to do that, we interpret the construction of the flex locus via two (admissible) covers as follows. For $(E, p_1, p_2, p_3) \in \BFF$, there exist two (admissible) covers $\pi$ and $\phi$ of degree three to rational curves, whose domain curves (with marked points $p_i$ and $q_j$) have stable model $(E, p_1, p_2, p_3)$ after forgetting the $q_j$, such that each $p_i + 2q_i$ is a fiber of $\pi$ and such that $p_1 + p_2 + p_3$ and $q_1 + q_2 + q_3$ are fibers of $\phi$, where the two $g^1_3$ giving $\pi$ and $\phi$ are linearly equivalent, i.e. they provide a $g^2_3$ that realizes $E$ as a plane cubic with the desired configuration.  

\subsection*{Intersection of $\BFF$ with $\Delta_{0; \{1,2,3\}}$}
Let $\Delta_{0; \{1,2,3\}}$ be the boundary divisor of $\BM_{1,3}$ whose general element parameterizes an elliptic curve $E$ union a rational component $R$ at a node $r$ where the three marked points are contained in $R$. Using the admissible cover description, one can easily check that a general element in $\BFF\cap \Delta_{0; \{1,2,3\}}$ has two types. 

\begin{enumerate}[(I)]

\item One case is that $E$ is determined as a double cover of $R$ branched at $p_1, p_2, p_3$ and the node $r$.  In this case, the (implicitly marked points) $q_1, q_2, q_3$ are the ramification points, i.e. $2q_i \sim 2r$ in $E$. We have $3r \sim p_1+p_2+p_3 \sim q_1 + q_2 + q_3 \sim p_i + 2q_i$, where $p_i = r$ if we blow down $R$. In other words, this case occurs when the three cocritical points collide at a flex point in the plane cubic model of $E$.    

\item Alternatively, $q_1, q_2, q_3$ can be contained in $R$. Then $R$ admits a triple cover of $\bbP^1$ such that 
$p_1 + 2q_1$, $p_2 + 2q_2$, $p_3 + 2q_3$ and $2r + t$ (for some $t$ in $R$) belong to the fibers of a $g^1_3$ on $R$, where $E$ admits a double cover induced by $2r$ together with a $\bbP^1$ attached to $R$ at $t$ to form the other part of the admissible cover. The other $g^1_3$ will also have $2r$ contained in a fiber. In this case the plane cubic model contracts $E$ and maps $R$ to a singular plane cubic with a cusp at $r$. To see this explicitly, consider a cuspidal model of \cite[(2.10)]{MMW} by taking $b(x) = 3$ and $c(x) = -4x^3$ therein. Then one can check that the cocritical line away from the cusp (the one denoted by $L'_2$ in \cite[(2.11)]{MMW}) cuts out three points $p_1, p_2, p_3$ that together with the cusp $r$ have two distinct cross-ratios in $R\cong \bbP^1$.  In other words, fixing $(p_1, p_2, p_3) = (0, 1, \infty)$ in $R$, let $c$ be the number of such special positions of $r$ in $R$ where the above configuration holds, and then $c = 2$. 
Later we will also provide a sanity check for $c=2$ by using test curves.  
\end{enumerate}

\subsection*{Intersection of $\BFF$ with $\Delta_{0; \{i, j\}}$} Let $\Delta_{0; \{i, j\}}$ be the boundary divisor of $\BM_{1,3}$ whose general element parameterizes an elliptic curve $E$ union a rational component $R$ at a node $r$ such that the marked points $p_i$ and $p_j$ are contained in $R$ and the remaining one is contained in $E$. One checks that a general element in $\BFF\cap \Delta_{0;\{i,j \}}$ parameterizes 
a one-marked irreducible nodal rational curve $N$ (as a degeneration of $E$) union $R$ at a node $r$. More precisely, suppose $N$ contains $p_1$ and $R$ contains $p_2, p_3$ for $\{i, j \} = \{2,3\}$. The plane cubic model contracts $R$ so that $p_2 = p_3 = r$. In this case the tangent line to $N$ at $r$ goes through $p_1$, the projection of $N$ from $p_1$ has another tangent line at $q_1$, and $q_2, q_3$ coincide with the non-separating node $n$ of $N$. The projection center $s$ is the intersection of the lines $\overline{p_1q_1}$ and $\overline{n r}$. 

\subsection*{Test curves}
Let $\Delta_{\irr}$ be the boundary divisor of $\BM_{1,3}$ whose general element parameterizes an irreducible rational nodal curve.  We use $\delta_{\bullet}$ to denote the  divisor class of each $\Delta_{\bullet}$. The rational Picard group of $\BM_{1,3}$ is generated by $\delta_{\irr}$, $\delta_{0; \{i,j\}}$, and $\delta_{0; \{1,2, 3\}}$. Suppose the divisor class of $\BFF$ in $\BM_{1,3}$ is 
$$  \BFF =  d_1 \delta_{\irr} + d_2 \big(\delta_{0; \{1,2\}} + \delta_{0; \{1,3\}} + \delta_{0; \{2,3\}}\big) + d_3 \delta_{0; \{1,2, 3\}}. $$
Based on the above descriptions, we will use a number of test curves to determine the coefficients $d_i$ and provide sanity checks. 

\begin{enumerate}[(i)]
\item
Take the family $B_1$ of a rational curve $R$ union a pencil of plane cubics at a node $r$ where $R$ contains $p_1, p_2, p_3$ in general position.  Then we have 
$$ B_1 \cdot \delta_{\irr} = 12, $$
$$ B_1 \cdot \delta_{0;\{ i,j\}} = 0, $$
$$ B_1 \cdot \delta_{0;\{ 1,2,3\}} = -1, $$
$$ B_1 \cdot \BFF= 12, $$
$$ 12 d_1 - d_3 = 12. $$ 
The intersection number with $\BFF$ is due to that the moduli of $E$ is determined by the double cover ramified over the $p_i$ and $r$ (described as type (I) in the intersection of $\BFF$ with $\Delta_{0;\{ 1,2,3\}}$) and a pencil of plane cubics mapping to $\BM_{1,1}$ has degree $12$. 

\item
Take the family $B_2$ by attaching a pencil of plane cubics to a rational curve $R$ at a node $r$ where $R$ contains $p_2, p_3$ and we choose a base point of the pencil to be $p_1$.  Then we have 
$$  B_2 \cdot \delta_{\irr} = 12, $$
$$  B_2 \cdot \delta_{0;\{2,3\}} = -1, $$
$$  B_2 \cdot \delta_{0;\{ 1,2,3\}} = B_3 \cdot \delta_{0;\{1,2\}} = B_3 \cdot \delta_{0;\{1,3\}} = 0, $$
$$  B_2 \cdot \BFF = 12, $$
$$ 12d_1 - d_2 = 12. $$
The intersection number with $\BFF$ comes from the $12$ rational nodal curves parameterized in the pencil.  

\item
Take the family $B_3$ arising from a general plane cubic $E$ with a general marked point $p_1$ and a varying line $L$ through $p_1$ that cuts out $p_2, p_3$ in $E$. In order to label $p_2$ and $p_3$, we need to pass to a double cover, i.e. $B_3\cong E$ is a double cover of $\bbP^1$ (parameterizing the varying lines through $p_1$).   Then we have 
$$  B_3 \cdot \delta_{\irr} = 0, $$
$$  B_3 \cdot \delta_{0;\{2,3 \}} = 4, $$
$$  B_3 \cdot \delta_{0;\{1,2 \}} =  B_4 \cdot \delta_{0;\{1,3 \}}  = 1, $$
$$  B_3 \cdot \delta_{0;\{1,2,3\}} = 0, $$
$$  B_3 \cdot \BFF = 24, $$
$$ 6d_2 = 24. $$ 
To see the intersection number with $\BFF$, note that there are four cocritical lines through $p_1$ that can cut out $p_1 + 2q_1$ and the Hessian ${\rm H}(E)$ of $E$ has degree three (see \cite[Section 2]{MMW}). Hence each of the four cocritical lines meets ${\rm H}(E)$ at three points, providing in total $12$ choices for the projection center $s$. The additional factor $2$ is due to the labeling of $p_2$ and $p_3$. 

Note that by now we can already determine the divisor class of $\BFF$ as claimed in Theorem~\ref{thm:divisor}. Nevertheless, we will work out a few more test curves to do a sanity check as well as to determine the special value $c$ appearing in the analysis for the intersection of $\BFF$ with $\Delta_{0; \{1,2,3 \}}$ of type (II).  

\item
Take the family $B_4$ arising from a rational curve $R$ attached to a general elliptic curve $E$ at a node $r$ where $R$ contains $(p_1, p_2, p_3) = (0, 1, \infty)$ and $r$ varies in $R$.  Then we have 
$$ B_4 \cdot \delta_{\irr} = 0, $$
$$ B_4 \cdot \delta_{0;\{i,j\}} = 1, $$
$$ B_4 \cdot \delta_{0;\{1,2,3\}} = -1, $$
$$ B_4 \cdot \BFF = 6 + c, $$
$$ 3d_2 - d_3 = 6 + c. $$ 
To see the intersection number with $\BFF$, first note that the four branched points of the double cover from $E$ to $\bbP^1$ induced by $2r$ has 
six distinct cross-ratios, which contributes $6$ by the description of type (I) for the intersection of $\BFF$ with $\Delta_{0;\{ 1,2,3\}}$. Moreover, 
when the varying node $r$ meets the $c$ special positions in the description of type (I) for the intersection of $\BFF$ with $\Delta_{0;\{ 1,2,3\}}$, we obtain another contribution equal to $c$. Therefore, we conclude that $c = 2$ (using $d_2 = d_3 = 4$). 

\item
Take the family $B_5$ arising from a general plane cubic $E$ with a flex point $p_1$ and a varying line $L$ through $p_1$ that cuts out $p_2, p_3$ in $E$. Then we have 
$$ B_5 \cdot \delta_{\irr} = 0, $$
$$ B_5 \cdot \delta_{0;\{2,3\}} = 3, $$
$$ B_5 \cdot \delta_{0;\{1,2\}} = B_5 \cdot \delta_{0;\{1,3\}} = 0, $$
$$ B_5 \cdot \delta_{0,\{1,2,3\}} = 1, $$ 
$$ B_5 \cdot \BFF = 16, $$
$$ 3d_2 + d_3 = 16. $$
Note that when $p_2$ and $p_3$ both coincide with $p_1$, we obtain a rational comonent $R$ union $E$ at a node $r$ such that $2r\sim 2p_1 \sim p_2 + p_3$ gives the same $g^1_2$ on $R$, whose elliptic partner (described as type (I) in the intersection of $\BFF$ with $\Delta_{0;\{ 1,2,3\}}$) can be avoided by choosing a general $E$. To see the intersection number with $\BFF$, note that 
the Hessian ${\rm H}(E)$ meets $E$ at the nine flexes including $p_1$. Hence each of the three (non-flex) cocritical lines through $p_1$ meets ${\rm H}(E)$ at two other points besides $p_1$, providing $6$ choices for the projection center $s$, and we also need to choose a labeling of $p_2$ and $p_3$. The additional $4$ is from the case when $q_1 = p_1$, i.e. the flex line through $p_1$ becomes a special cocritical line. 

\item
Take the family $B_6$ by varying a line through a general point away from a general plane cubic $E$ to cut out $p_1, p_2, p_3$ in $E$ (after making a base change of degree six in order to label the $p_i$).  Then we have 
$$ B_6 \cdot \delta_{\irr} = 0, $$
$$ B_6 \cdot \delta_{0;\{i,j\}} = 18, $$
$$ B_6 \cdot \delta_{0;\{1,2,3 \}} = 0, $$
$$ B_6 \cdot \BFF = 72, $$
$$ 18d_2 = 72. $$
The intersection number with $\BFF$ is due to the base change and the fact that the satellite Cayleyan parameterizing cocritical lines of $E$ in the dual plane has degree $12$ (see \cite[Section 2]{MMW}). 

\item
Take the family $B_7$ by varying $p_1$ along a general elliptic curve $E$ attached to a rational curve $R$ containing $p_2$ and $p_3$. Then we have 
$$ B_7 \cdot \delta_{\irr} = 0, $$
$$ B_7\cdot \delta_{0;\{2,3\}} = -1, $$
$$ B_7\cdot \delta_{0;\{1,2\}} = B_7\cdot \delta_{0;\{1,3\}} = 0, $$
$$ B_7\cdot \delta_{0;\{1,2,3\}} = 1, $$
$$ B_7 \cdot \BFF = 0, $$
$$ d_3 - d_2 = 0. $$
\end{enumerate}

We can also verify intersection transversality for the test curves with $\BFF$ without much difficulty. For instance, $B_6$ is a freely moving curve in $\BM_{1,3}$ by varying the moduli of $E$ and the base point of the lines that cut out the marked points.  Hence a general choice of $B_6$ meets $\BFF$ transversally at its generic points. Similarly one can check this way for test curves contained in the boundary. For instance, $B_1$ is a freely moving curve in $\Delta_{0; \{1,2,3\}}$ by varying the moduli of $E$ and the marked points in $R$. 

In summary, the above test curves determine and cross check the divisor class of $\BFF$ as claimed in~\eqref{eq:F}.   

\section{Intersection theory on the gothic locus}
\label{sec:gothic}

Recall that $\bbP\QQ\BFF\subset \bbP\BQQ_{1,3}$ is the closure of the locus $(E, [q], p_1, p_2, p_3)$ where $(E, p_1, p_2, p_3) \in \calF$ and 
${\rm div}(q) = z_1 + z_2 + z_3 - p_1 - p_2 - p_3$ with $z_1 + z_2 + z_3$ as a fiber of the $g^1_3$ given by the projection of $E$ from $s$.  Note that here the $p_i$ are labeled but we do not label the $z_i$.

Let $T\subset \bbP\QQ\BFF$ be the closure of the locus where ${\rm div}(q) = 2q_1 - p_2 - p_3$, i.e. the original zero divisor $z_1+z_2+z_3$ of $q$ is chosen to be the special fiber $p_1 + 2q_1$. Note that $T$ can be realized as a (rational) section of $f\colon \bbP\QQ\BFF\to \BFF$ via 
$(E, p_1, p_2, p_3) \mapsto (E, [q], p_1, p_2, p_3)$ where ${\rm div}(q) = 2q_1 - p_2 - p_3$, which implies that $f_{*} T = \BFF$. 

\begin{lemma}
The following relation of divisor classes holds on $\bbP\QQ\BFF$: 
\begin{eqnarray}
\label{eq:eta-T}
\eta  =  \psi_{p_1} - T.  
\end{eqnarray}
\end{lemma}

\begin{proof}
Let $\pi\colon \calE \to \bbP\QQ\BFF$ be the universal curve.  Then we have the relation 
\begin{eqnarray}
\label{eq:relation}
 \pi^{*}\eta =  2c_1(\omega_\pi) + P_1+P_2+P_3- Z - V 
 \end{eqnarray}
where $\omega_\pi$ is the relative dualizing line bundle, each $P_i\subset \calE$ is the section corresponding to the marked point $p_i$, $Z\subset \calE$ is the closure of the locus of the zeros $z_i$, and $V\subset \calE$ is the vertical vanishing divisor arising from components of reducible curves $E$ on which $q$ is identically zero (see~\cite[Section 3]{C}). Intersecting both sides of~\eqref{eq:relation} with $P_1$ and applying $\pi_{*}$, we conclude that 
\begin{eqnarray}
\label{eq:push}
\eta  =  \psi_{p_1} - \pi_{*}(Z\cdot P_1) - \pi_{*}(V\cdot P_1).  
\end{eqnarray}

We first analyze the intersection of $Z$ and $P_1$.  In this case $z_1+z_2+z_3 = p_1 + 2q_1$, i.e. ${\rm div}(q) = 2q_1 - p_2 - p_3$. It follows that 
$\pi_{*}(Z\cdot P_1) = T$. 

Next we show that $\pi_{*}(V\cdot P_1)$ is the zero divisor class. For this we argue using the level graph terminology as described in~\cite{BCGGM1, BCGGM2} for the incidence variety compactification (IVC) over $\bbP\QQ\BFF$. Suppose $\Gamma$ is a level graph where $p_1$ is contained in the lower level (i.e. a vanishing component of $q$ contains $p_1$). Observe that if $z_1, z_2, z_3$ are all contained in the lower level of $\Gamma$, then their variation is not recorded in the stable quadratic differential $q$, as $q$ is identically zero on the lower level. Hence in this case the corresponding locus has codimension higher than one in $\bbP\QQ\BFF$, which contributes zero as a divisor class (i.e. it gets contracted from the IVC to $\bbP\QQ\BFF$). 
  
Consider curves in $\BFF\cap \Delta_{0;\{1,2,3\}}$ where the rational component $R$ contains $p_1, p_2, p_3$.  
 Then $q_1, q_2, q_3$ are either all contained in $E$ or all contained in $R$ according to the two types of the intersection of $\BFF$ with $\Delta_{0;\{1,2,3\}}$. Suppose $q_1, q_2, q_3$ are all contained in $E$. Then the admissible $g^1_3$ has a semistable rational component $B$ connecting $E$ and $R$.  If $z_1, z_2, z_3$ are all contained in $B$, since $B$ is in lower level, this locus drops dimension from the IVC to $\bbP\QQ\BFF$. If $z_1, z_2$ are in $E$ and $z_3$ is in $R$, then the nodes are all poles of order two (i.e. horizontal in the level graph), hence this locus is irrelevant as it does not admit a vanishing component. Suppose $q_1, q_2, q_3$ are all contained in $R$. Then the admissible $g^1_3$ has a rational component $C$ joining $R$ such that the union of $C$ and $E$ gives the other part of the triple cover.  If $z_1, z_2, z_3$ are all contained in $R$, then $R$ is in lower level compared to $E$, hence this locus drops dimension from the IVC to $\bbP\QQ\BFF$. If $z_1, z_2$ are in $E$ and $z_3$ is in $C$, then after contracting the unstable component $C$, $z_3$ goes to $R$, and the node joining $R$ to $E$ is horizontal, hence this locus is irrelevant as it does not admit a vanishing component. 
 
Finally consider curves in $\BFF\cap \Delta_{0;\{ i, j\}}$ where the rational component $R$ contains $p_i, p_j$ and the irreducible rational nodal  component $N$ contains the remaining marked pole. In this case $N$ and $R$ are on the same side of the $g^1_3$.  Note that $z_1 + z_2 + z_3$ is given by the projection of $N$ from $s$.  If $z_1 + z_2 + z_3 \neq p_1 + p_2 + p_3$, then the node $r$ joining $R$ and $N$ is horizontal, hence the locus is irrelevant as it does not admit a vanishing component. If $z_1 + z_2 + z_3 = p_1 + p_2 + p_3$, then the stable quadratic differential $q$ has no pole in $N$, hence $N$ is on top level. Suppose $z_1 = p_i$ in $N$ and $z_2 + z_3 = p_j + p_k$ are in lower level for $\{i,j,k \} = \{1,2,3\}$.  Then $z_1$ is not varying on top level but $z_2, z_3$ vary in lower level in the $g^1_2$ determined by $2r\sim p_j + p_k$, hence the locus drops dimension from the IVC to $\bbP\QQ\BFF$.  

In summary, we have shown that $\pi_{*}(Z\cdot P_1) = T$ and $\pi_{*}(V\cdot P_1) = 0$. Hence the desired Equation~\eqref{eq:eta-T} follows from~\eqref{eq:push}.  
\end{proof}

Next we study the intersection $\eta T$.  Let $S$ be the IVC compactification of the locus $(E, [q], p_1, p_2, p_3, q_1)$ where 
$(E, [q], p_1, p_2, p_3) \in \calF$ with ${\rm div}(q) = 2q_1 - p_2 - p_3$. In other words, the signature of the ambient stratum of the locus is $(2,-1,-1,0)$ where $p_1$ is regarded as an ordinary marked point.  Let $g\colon S\to T$ be the map that forgets $q_1$, and $g$ is a birational morphism.  

Let $\Lambda_1$ be the level graph parameterizing an irreducible rational nodal curve $(N, p_3, q_3, q_1 =q_2 =n)$ union a rational curve $(R, p_1, p_2)$ at a node $r$ such that $q|_R = 0$ and $(q|_{N'}) = -p_3 - r - n_1 - n_2$, where $N' \cong \bbP^1$ is the normalization of $N$ after blowing up the node $n$ to $n_1, n_2$, and the resulting exceptional rational component carries $2q_1 - 3n_1 - 3n_2$ for the corresponding twisted quadratic differential. Let $\Lambda_2$ and $\Lambda_3$ be the (isomorphic) level graphs parameterizing an elliptic curve $E$ union each one of the two special rational components $(R, p_1, p_2, p_3, q_1, q_2, q_3)$ in the type (II) intersection of $\BFF$ with $\Delta_{0;\{1,2,3 \}}$ where $q|_E = (dz)^2$ and $q|_{R} = 0$. Denote by $\alpha_{\Lambda_i}$ the divisor class of the corresponding locus of $\Lambda_i$ in $S$, and by $\beta_{\Lambda_i}$ the corresponding image class in $T$.

\begin{lemma}
The following relation of cycle classes holds: 
\begin{eqnarray}
\label{eq:eta-S-T}
 \eta T = \psi_{p_2} T - \beta_{\Lambda_1} - 2 \beta_{\Lambda_2} - 2\beta_{\Lambda_3}. 
 \end{eqnarray}
\end{lemma}

\begin{proof}
Using a similar analysis as in the proof of Equation~\eqref{eq:eta-T}, we can obtain the following relation of divisor classes on $S$: 
\begin{eqnarray}
\label{eq:eta-S}
 g^{*}\eta = g^{*}\psi_{p_2}  - \alpha_{\Lambda_1} - 2 \alpha_{\Lambda_2} - 2  \alpha_{\Lambda_3}
 \end{eqnarray}
 where the coefficient of $\alpha_{\Lambda_i}$ for $i = 2, 3$ has a factor $2$ from the twisting order of the edge in the level graph. Then Equation~\eqref{eq:eta-S-T} follows from applying $g_{*}$ to~\eqref{eq:eta-S}. 
\end{proof}

Finally we can work out the desired intersection numbers in~\eqref{eq:L}. 

\begin{proof}[Proof of Theorem~\ref{thm:L}]
Recall the map $f\colon \bbP\QQ\BFF\to \BFF$. We have $f_{*} T = \BFF$, $f^{*}\lambda_1 = \lambda_1$, and $f^{*} \psi_{p_i} = \psi_{p_i}$. Moreover, it is easy to see that $\eta \beta_{\Lambda_1} = \frac{1}{2}$ and $\eta \beta_{\Lambda_2} = \eta \beta_{\Lambda_3} = \frac{1}{12}$. Also, $\psi_{p_i} \beta_{\Lambda_j} = 0$ 
for any $i \in \{1,2 \}$ and $j\in \{ 1,2,3\}$ as $p_1$ and $p_2$ are marked points in a rational component with fixed moduli in all $\Lambda_i$, $\lambda_1 \beta_{\Lambda_1} = 0$ as both components in $\Lambda_1$ are rational, and $\lambda_1 \beta_{\Lambda_2} = \lambda_1 \beta_{\Lambda_3}  = \frac{1}{24}$. 
Then using Equations~\eqref{eq:eta-T} and~\eqref{eq:eta-S-T} we conclude that  
\begin{eqnarray*}
f_{*} (\eta^3) & = & f_{*} \Big( \eta^2 \psi_{p_1} - \eta^2 T \Big) \\ 
& = & f_{*} \Big(- \eta \psi_{p_2} T + \eta \beta_{\Lambda_1} + 2 \eta \beta_{\Lambda_2} + 2 \eta \beta_{\Lambda_3} -\eta \psi_{p_1} T + \eta \psi_{p_1}^2\Big) \\
& = & \frac{5}{6} + f_{*} \Big(  - \psi_{p_2}^2 T + \psi_{p_2} \beta_{\Lambda_1} + 2 \psi_{p_2} \beta_{\Lambda_2} + 2 \psi_{p_2} \beta_{\Lambda_3} \\
&  & -\psi_{p_1}\psi_{p_2}T + \psi_{p_1} \beta_{\Lambda_1} + 2 \psi_{p_1} \beta_{\Lambda_2} + 2 \psi_{p_1} \beta_{\Lambda_3} + \eta \psi_{p_1}^2  \Big) \\
& = &  \frac{5}{6} - \psi_{p_2}^2  \BFF - \psi_{p_1}\psi_{p_2} \BFF - \psi_{p_1}^2 \BFF 
\end{eqnarray*}
and 
\begin{eqnarray*}
f_{*} (\eta^2 \lambda_1) & = & f_{*} \Big(-\eta \lambda_1 T + \eta \lambda_1 \psi_{p_1}  \Big) \\ 
& = & f_{*}\Big(-\lambda_1\psi_{p_2}T + \lambda_1 \beta_{\Lambda_1} + 2 \lambda_1 \beta_{\Lambda_2} +2 \lambda_1 \beta_{\Lambda_3} + \eta\lambda_1 \psi_{p_1}\Big)   \\
& = & -\lambda_1\psi_{p_2}  \BFF  + \frac{1}{6} - \lambda_1\psi_{p_1} \BFF. 
\end{eqnarray*}

On $\BM_{1,3}$ we have the following intersection numbers: 
$$ \lambda_1 \psi_{p_1} \delta_{\irr} = 0, \quad \lambda_1 \psi_{p_1} \delta_{0; \{1,2\}}  =  \lambda_1 \psi_{p_1} \delta_{0; \{1,3\}} = 0, $$
$$ \lambda_1 \psi_{p_1} \delta_{0; \{2,3\}}  = \frac{1}{24}, \quad \lambda_1 \psi_{p_1} \delta_{0; \{1,2,3\}}  = \frac{1}{24}, $$
$$\psi_{p_1}^2 \delta_{\irr} = \frac{1}{2}, \quad \psi_{p_1}^2 \delta_{0; \{1,2\}}  =  \psi_{p_1}^2 \delta_{0; \{1,3\}} = 0, $$
$$ \psi_{p_1}^2 \delta_{0; \{2,3\}}  = \frac{1}{24}, \quad \psi_{p_1}^2 \delta_{0; \{1,2,3\}} = 0, $$
$$  \psi_{p_1} \psi_{p_2} \delta_{\irr} = 1, \quad \psi_{p_1}\psi_{p_2} \delta_{0; \{1,2,3\}} = 0, $$
$$\psi_{p_1}\psi_{p_2} \delta_{0; \{1,2\}}  = \psi_{p_1}\psi_{p_2} \delta_{0; \{1,3\}} =  \psi_{p_1}\psi_{p_2} \delta_{0; \{2,3\}} = 0. $$

Combining the above with the divisor class of $\BFF$ in~\eqref{eq:F}, it follows that 
$$ \psi_{p_1}^2 \BFF  =  \psi_{p_2}^2 \BFF = \frac{5}{6}, \quad \psi_{p_1}\psi_{p_2} \BFF = \frac{4}{3}, $$
$$ \lambda_1 \psi_{p_1} \BFF = \lambda_1 \psi_{p_2} \BFF = \frac{1}{3}. $$

Finally we obtain that 
$$ \int_{\bbP\QQ\BFF} \eta^3 = -\frac{13}{6} \quad {\rm and}\quad \int_{\bbP\QQ\BFF} \eta^2 \lambda_1  = -\frac{1}{2}, $$
thus verifying~\eqref{eq:L}. 
\end{proof}


\end{document}